\documentclass[11 pt,amstex,3p,times]{article}
\usepackage{hyperref}
\usepackage{enumerate,fullpage}
\usepackage{amssymb,amsmath,graphicx,amsthm}
\usepackage{color}

\newcommand{\inpro}[2]{\left\langle{#1},{#2}\right\rangle}

\newcommand{\be}{\begin{eqnarray*}}
	\newcommand{\en}{\end{eqnarray*}}
\newcommand{\bes}{\begin{eqnarray}}
\newcommand{\ens}{\end{eqnarray}}

\newcommand{\lf}{\left}
\newcommand{\rt}{\right}

\def  \kg {L^\infty\lf(0,T; 		 \mathcal{Z}_{\gamma, TA_1 }(\Omega)\rt)}
\usepackage{epic, curves}
\def\nn{\nonumber}

\newtheorem{theorem}{Theorem}[section]

\newtheorem{corollary}{Corollary}[section]

\newtheorem{lemma}{Lemma}[section]

\def\bq{\begin{equation}}
\def\eq{\end{equation}}
\def\bqq{\begin{eqnarray*}}
	\def\eqq{\end{eqnarray*}}

\def\nn{\nonumber}

\title{\bf A random regularized approximate  solution of   the  inverse  problem for  the Burgers' equation }
\author{ Erkan Nane $^1$ \footnote{E. Nane: \url{ezn0001@auburn.edu}},   Nguyen Hoang Tuan$^2$, Nguyen  Huy Tuan $^2$ \footnote{Corresponding author: \url{nguyenhuytuan@tdt.edu.vn }},  \\\\
	\small $^1$ Department of Mathematics and Statistics, Auburn University, Auburn, USA \\
\small $^{2}$ Applied Analysis Research Group,
Faculty of Mathematics and Statistics,\\
\small Ton Duc Thang University, Ho Chi Minh City, Vietnam
 \\ \\
}
\begin{document}
	\date{}
	\maketitle
	
	\begin{abstract}
		In this paper, we  find a  regularized approximate solution  for   an   inverse  problem for the Burgers' equation. The solution of  the inverse problem for the  Burgers' equation  is ill-posed, i.e., the solution  does not depend continuously on the data. The approximate solution is the solution of a regularized equation with randomly perturbed coefficients and randomly perturbed final value and source functions.
		   To find the regularized solution, we use the modified quasi-reversibility method  associated with the truncated expansion method  with nonparametric
		regression. We also investigate the convergence
		rate.
	\end{abstract}


\section{Introduction}

	In this work, we consider the backward in time problem for  { \bf 1-D Burgers' equation}
	\begin{equation}
		\label{problem4444}
		\left\{\begin{array}{l l l}
			{\bf u}_t -(A(x,t)	{\bf u}_x)_x & =	{\bf u} 	{\bf u}_x+G(x,t), & \qquad (x,t) \in \Omega\times (0,T),\\
			{\bf u}(x,t)&=0, & \qquad x \in \partial {\Omega},\\
			{\bf u}(x,T) & = H(x), & \qquad x \in {\Omega},
		\end{array}
		\right.
	\end{equation}	
	where $\Omega =(0,\pi)$.
	The Burgers equation is a fundamental partial differential equation occurring in various areas of applied mathematics, such as fluid mechanics, nonlinear acoustics, gas dynamics, traffic flow \cite{Ku}.

{


One  can see that the term  $(A(x,t)    {\bf u}_x)_x $ is $\Delta {\bf u}={\bf u}_{xx}$ if  $A=1$.
However,  one  can not use spectral methods to study the operator  $(A(x,t)    {\bf u}_x)_x$
So, the problem is more difficult.
The second observation is that for the equation   ${\bf u}_t- (A(x,t)    {\bf u}_x)_x = f({\bf u}, {\bf u}_x)$
when  $A(x,t)$ is deterministic  and $f({\bf u}, {\bf u}_x) = f({\bf u})$, the problem is a consequence of Theorem 4.1 in our recent paper \cite{kirane-nane-tuan-1}.
However, if   $ A(x,t)$ is randomly  perturbed and $f({\bf u}, {\bf u}_x)$ depends on ${\bf u}$ and  ${\bf u}_x$ then the problem is more challenging.


Until now, the deterministic  Burgers' equation with the randomly perturbed case have not been  studied.
Hence, the paper is the  first study of  Burgers' equation  backward in time. The  inclusion of the gradient term  in  $u u_x$ in the right hand side  of the Burgers' equation    makes the Burgers' equation more difficult to  study. We need  to find  an approximate function for ${\bf u} {\bf u} _x$.
This task  is nontrivial.

This paper is a continuation of our  study of  backward problems  in the two recent papers \cite{kirane-nane-tuan-1, kirane-nane-tuan-2}. In those papers the equations did not have  random coefficients in the main equations.
The paper \cite{kirane-nane-tuan-1} does not  consider the random  operator. The paper \cite{kirane-nane-tuan-2} considers the  simple coefficient  $A(x,t)=A(t)$ and the source function is  ${\bf u}-{\bf u}^3$.
Hence, one can see that the Burgers' equation considered here  is more difficult since the gradient term in the right hand side and the coefficient $A(x,t)$  depends on   both $x$ and $t$.

}

It is known that the backward  problem mentioned above is ill-posed in general \cite{Ku}, i.e., a solution does not always exist. When the solution exists, the solution does not depend continuously on the given initial data. In fact, from a small noise
of a physical measurement, the corresponding solution may have a large error. This
makes the numerical computation  troublesome. Hence a regularization is required. It is well-known that  there are some difficulties to  study  the nonlocal Burger's equation.
First, by the  given form of  coefficient $A(x,t)$ in the main equation \eqref{problem4444},  the solution of Problem \eqref{problem4444} can not be transformed into a nonlinear integral equation. Hence, classical spectral method cannot  be  applied.
The second  thing that  makes the Burger's equation more difficult to study is the gradient term ${\bf u}_x$ in the right hand side.  Until now, although there are limited  number of works  on  the backward problem  for Burgers' equation \cite{Carasso,Hao},  there are no results for regularizing the problem.

As is well-known, measurements always are given at a discrete set of points
and contain errors. These errors may be generated from controllable sources or uncontrollable sources. In the first case, the error is often deterministic.  If the errors are generated from uncontrollable sources as wind, rain, humidity, etc., then the model is random. Methods used for the deterministic cases cannot be  applied directly to the random  case.

In this paper,  we consider the following model  as follows
\begin{equation}
\widetilde H(x_k)=H(x_k)+\sigma_k\epsilon_k, \quad  \widetilde G_k(t)= G(x_k,t)+ \vartheta\xi_{k}(t),\quad \text{for} \quad  k=\overline{1,n},  \label{noise}
\end{equation}
and
\begin{equation}
\widetilde A_k(t)= A(x_k,t)+  { \overline \vartheta} { \xi}_{k}(t),\quad \text{for} \quad  k=\overline{1,n}.  \label{noise1}
\end{equation}
where $x_k = \pi\dfrac{2k-1}{2n}$
and  $\epsilon_k$ are unknown independent random errors. Moreover,   $\epsilon_{k}\sim \mathcal{N}(0,1)$,  and $\sigma_{k} ,~\vartheta, \overline \vartheta$ are unknown positive constants which are bounded by a positive constant $V_{max}$, i.e., $0 \leq \sigma_{k} < V_{\text{max}}$ for all $k=1,\cdots, n$.  $\xi_{k}(t)$ are Brownian motions. The noises $ \epsilon_{k}, \xi_{k}(t)$ are mutually independent.  	Our task is reconstructing the initial data ${ u} (x, 0)$.

We next want to mention about the organization of the paper and  our methods in this paper.  We prove some preliminary results in section \ref{section2}. We state and prove our main result in section \ref{section3-main}.  The existence and uniqueness of   solution of equation \eqref{problem4444} is an open problem, and we do not investigate this problem here. For inverse problem, we assume that the solution of  the Burgers'
equation \eqref{problem4444} exists. In this case  its  solution  is not stable. In this paper  we establish an approximation of the  backward in time problem for  {  1-D Burgers' equation} \eqref{problem4444} with the solution of a regularized equation with randomly perturbed equation \eqref{reguburger}. The random perturbation in equation \eqref{reguburger} is explained in  equations \eqref{noise}, \eqref{noise1}, \eqref{reguburger} and \eqref{randomly-perturbed-operator}.

\section{Some Notation}\label{section2}
We first  introduce notation, and then  state the first set of our  main results in this paper.
We define fractional powers of the Neummann-Laplacian.
{
	\begin{equation}
		Af(x):= -\Delta f(x)=-\frac{\partial^2f(x)}{\partial x^2}.
	\end{equation}
	Since $A$ is a linear densely defined self-adjoint and positive definite
	elliptic operator on the connected bounded domain  $\Omega=(0,\pi) $ with
	{Dirichlet  boundary condition}, the eigenvalues of $A$ satisfy
	\[
	\lambda_0=0 < \lambda_1 \le \lambda_2 \le \lambda_3 \le \cdots \le \lambda_p\le \cdots
	\]
	with  $\lambda_p=p^2  \to \infty $ as $p \to \infty$.
	The corresponding eigenfunctions are denoted respectively by
	{$\varphi_{p}(x)=\sqrt{\frac{2}{\pi}}\sin(px) $.}
	Thus the eigenpairs $(\lambda_p,\phi_p)$,
	$p=0,1,2,...$, satisfy
	\[
	\begin{cases}
	A \varphi_{p}(x)
	=
	-\lambda_p \phi_{p}(x),
	\quad & x \in \Omega \\
	\partial_x\phi_{p}(x)
	=
	0,
	\quad & x\in \partial \Omega.
	\end{cases}
	\]
	The functions $\varphi_p$ are normalized so that
	$\{\phi_{p}\}_{p=0}^\infty$ is an orthonormal basis of $L^2(\Omega)$.\\
	Defining
	\[
	H^{\gamma}(\Omega)
	=
	\Bigg\{ v \in L^{2}(\Omega) : \sum\limits_{p=0}^{\infty}
	\lambda_{p}^{2\gamma} |\inpro{v}{\phi_{p}}|^{2} <  + \infty \Bigg\},
	\]
	where $\inpro{\cdot}{\cdot}$ is the inner product in $L^{2}(\Omega)$, then
	$H^{\gamma}(\Omega)$ is a Hilbert space equipped with norm
	\[
	\|v\|_{H^{\gamma}(\Omega)}
	=
	\left(\sum\limits_{p=1}^{\infty}
	\lambda_{p}^{2\gamma} |\inpro{v}{\phi_{p}}|^{2}\right)^{1/2}.
	\]
	First, we  state  following Lemmas that will be used in this paper
	
	\begin{theorem}[Theorem 2.1 in \cite{kirane-nane-tuan-1}]  \label{theorem2.1}
		Define the set $	\mathcal{W}_{   \beta_{ n}}$ for any $ n \in \mathbb{N}$
		\begin{equation}
			\mathcal{W}_{   \beta_{ n}} = \Big\{ { p}  \in \mathbb{N}  : |{ p}| \le \sqrt{ \beta_{ n} } \Big\}
		\end{equation}
		where $\beta_{ n}$  satisfies
		\begin{equation*}
			\lim_{|{ n}| \to +\infty} \beta_{ n}=+\infty.
		\end{equation*}
		For a given ${ n}$ and $\beta_{ n}$ we define  functions that are approximating  $H, G$ as follows
		\begin{equation}
			\widehat H_{\beta_{ n}}({ x}) = \sum_{{ p} \in  	\mathcal{W}_{  \beta_{ n}}   } \Bigg[\frac{\pi}{n} \sum_{k=1}^n \widetilde D_{k} \psi_{ p}	( x_k) \Bigg] \psi_{ p}({ x}),~~
			\widehat G_{\beta_{ n}}({ x},t) = \sum_{{ p} \in  	\mathcal{W}_{  \beta_{ n}}   } \Bigg[\frac{\pi}{n} \sum_{k=1}^n \widetilde G_{k}(t) \psi_{ p}	( x_k) \Bigg] \psi_{ p}({ x}).
		\end{equation}
	 Let us choose $\mu_0 > \frac{1}{2}$.
		If  $H \in \mathcal{H}^{\mu_0}(\Omega) $ and $G \in L^\infty (0,T;\mathcal{H}^{\mu_0}(\Omega) )$ then the following estimates hold
		\begin{eqnarray}
			\begin{aligned}
				&	{\bf E}\Big\| \widehat H_{\beta_{ n}} -H  \Big\|^2_{L^2(\Omega)} \le  \overline C (\mu_0, H)  \sqrt{ \beta_{ n} }  n^{-4\mu_0} +{ 4\beta_{ n}^{-\mu_0}} \Big\| H \Big\|^2_{\mathcal{H}^{\mu_0}(\Omega) },\nn\\
				&	{\bf E}\Big\| \widehat G_{\beta_{\bf n}}(.,t) -G(.,t)  \Big\|^2_{L^\infty(0,T; L^2(\Omega))} \le  \overline C (\mu_0, G)  \sqrt{ \beta_{ n} }  n^{-4\mu} +{ 4\beta_{ n}^{-\mu_0}} \Big\| G \Big\|^2_{L^\infty (0,T;\mathcal{H}^{\mu_0}(\Omega) ) }	,
			\end{aligned}
		\end{eqnarray}	
		where
		\begin{equation*}
			\overline C (\mu_0, H)= 8\pi  V_{max}^2  \frac{2 \pi^{1/2}}{  \Gamma(1/2)} + \frac{ 16 \mathcal{C}^2 \mu_0 \pi^{1/2}}{  \Gamma(1/2)}  \Big\| H \Big\|^2_{\mathcal{H}^{\mu_0}(\Omega) }.
		\end{equation*}
		and
			\begin{equation*}
			\overline C (\mu_0, G)= 8\pi  V_{max}^2  \frac{2 \pi^{1/2}}{  \Gamma(1/2)} + \frac{ 16 \mathcal{C}^2 \mu \pi^{1/2}}{  \Gamma(1/2)}  \Big\| G \Big\|^2_{L^\infty (0,T;\mathcal{H}^{\mu_0}(\Omega) ) }.
			\end{equation*}		
	\end{theorem}
	
	\begin{corollary}[Corollary 2.1 in \cite{kirane-nane-tuan-1}]  \label{corollary2.1}
		Let  $H, G$ be as in  Theorem \eqref{theorem2.1}. Then  the term
		$ {\bf E}	\Big\| \widehat H_{\beta_{ n}}-H \Big\|_{L^2(\Omega)}^2+T {\bf E}	\Big\| \widehat G_{\beta_{ n}}-G \Big\|_{L^\infty(0,T;L^2(\Omega))}^2$
		is of order
		\begin{equation*}  \label{veryimportant}
			\max \Bigg( \sqrt{ \beta_{ n} }  n^{-4\mu_0} , ~\beta_{ n}^{-\mu_0}  \Bigg).
		\end{equation*}
	\end{corollary}

	\begin{lemma}
		Define the following space of functions
		\begin{equation} \label{definitionspace}
		\mathcal{Z}_{\gamma, B }(\Omega):=\Bigg\{ f \in L^2(\Omega),  \sum_{{ p} \in \mathbb{N}  }   { p}^{2+2\gamma} e^{2B { p}^{2}}  \big\langle f,\psi_{ p}\big\rangle_{L^2(\Omega)}^2 <+\infty \Bigg\},
		\end{equation}
		for any $\gamma \ge 0$ and $B\geq 0$. Define   also the operator  ${\bf P}= A_1 \Delta $  and
		$	{\bf P}_{\rho_{ n}}$   is defined as follows
		\begin{align}
		{\bf P}_{\rho_{ n}}(v)  &=A_1 \sum_{ { p} \le   \sqrt{\frac{\rho_{\bf n}} {  A_1 }  }  }^\infty       { p}^{2{}} \big\langle v(x),\psi_{ p}\big\rangle_{L^2(\Omega)} \psi_{ p},
		\end{align}
		for any function $v \in L^2(\Omega)$.  Then for  any $v \in L^2(\Omega)$
		\begin{align}
		\|{ \bf P}_{\rho_{\bf n}}(v)\|_{L^2(\Omega)} \le \rho_{ n} \|v\|_{L^2(\Omega)},  \label{estimateP1}
		\end{align}
		and for $v \in 		\mathcal{Z}_{\gamma, TA_1 }(\Omega) $ then
		\begin{align}
		\| {\bf P}v-{\bf P}_{\rho_{ n}}v\|_{L^2(\Omega)} \le  A_1 \rho_{ n}^{-\gamma}  e^{-T \rho_{ n}  }   \|v\|_{ \mathcal{Z}_{\gamma, TA_1 }(\Omega)  }. \label{estimateP2}
		\end{align}
		
	\end{lemma}
	\begin{proof}
		First, for any $v \in L^2(\Omega)$, we have
		\begin{align}
		\|{\bf P}_{\rho_{\bf n}}(v)\|_{L^2(\Omega)}^2 &= A_1^2 \sum_{ { p} \le   \sqrt{\frac{\rho_{n}} {  A_1 }  }  }^\infty    { p}^{4} \big\langle v(x),\psi_{ p}\big\rangle^2_{L^2(\Omega)} \nn\\
		&\le \rho_{ n}^2   \sum_{ { p} \le   \sqrt{\frac{\rho_{ n}} {  A_1 }  }  }^\infty \big\langle v(x),\psi_{ p}\big\rangle^2_{L^2(\Omega)}= \rho_{ n}^2  \|v\|_{L^2(\Omega)}^2,
		\end{align}
		and
		\begin{align}
		\| {\bf P}v -{\bf P}_{\rho_{ n}}(v)\|_{L^2(\Omega)}^2 &= A_1^2 \sum_{ { p} >   \sqrt{\frac{\rho_{ n}} {  A_1 }  }  }^\infty    { p}^{-4\gamma}  e^{-2TA_1{| p|}^2 } { p}^{4+4\gamma} e^{2TA_1{ p}^2 }  \big\langle v(x),\psi_{ p}\big\rangle^2_{L^2(\Omega)}\nn\\
		& \le  A_1^2 \rho_{ n}^{-2\gamma}  e^{-2TA_1 \rho_{ n}  } \sum_{ { p} >   \sqrt{\frac{\rho_{ n}} {  A_1 }  }  }^\infty  { p}^{4+4\gamma} e^{2TA_1{ p}^2 }    \big\langle v(x),\psi_{ p}\big\rangle^2_{L^2(\Omega)}\nn\\
		&=  A_1^2 \rho_{ n}^{-2\gamma}  e^{-2T \rho_{ n}  }   \|v\|_{ \mathcal{Z}_{\gamma, TA_1 }(\Omega)  }^2.
		\end{align}
		
	\end{proof}

	Now, we can assume that $\widehat A_{\rho_n}(x,t), A(x,t) \le A_0$ for all $(x,t ) \in \Omega \times (0,T)$  and we choose $A_1>A_0$. We describe our regularized problem by defining the following problem
	\begin{align}  \label{reguburger}
	\left\{ \begin{gathered}
		\frac{\partial \widetilde U_{\rho_n, \beta_n}}{\partial t}-\left( \widehat A_{\beta_n}(x,t) \frac{\partial \widetilde U_{\rho_n, \beta_n}}{\partial x}\right)_x -{\bf P}\widetilde U_{\rho_n, \beta_n} +{\bf P}_{\rho_n}\widetilde U_{\rho_n,\beta_n} = 	\\
		\quad \quad \quad \quad \quad \quad  ={\bf \overline F^0}_{\widehat Q_n} \left( \widetilde U_{\rho_n,\beta_n}, \frac{\partial \widetilde U_{\rho_n,\beta_n}}{\partial x} \right) +\widehat G_{\rho_n}(x,t) ,~~0<t<T, \hfill \\
		{\widetilde U_{\rho_n,\beta_n} }(x,t)= 0,~~x \in \partial \Omega,\hfill\\
		{\widetilde U_{\rho_n,\beta_n} }(x,T)=\widehat H_{\beta_n}(x). \hfill\\
	\end{gathered}  \right.
	\end{align}
	Here $	\widehat A_{\beta_n}$ is defined by
	\begin{equation}\label{randomly-perturbed-operator}
		\widehat A_{\beta_n}(x,t) = \sum_{p \in 	\mathcal{W}_{  \beta_{ n}}  } \Bigg[\frac{\pi}{n } \sum_{k=1}^{n} 	\widetilde A_{k}(t)\psi_p	( x_{k}) \Bigg] \psi_p(x)
	\end{equation}
	where $\psi_p(x)= \sqrt{\frac{2}{\pi}} \sin(px)$. Noting as above, 		 the function $F({\bf u} , {\bf u}_x )= {\bf u} {\bf u_x}$ in the first equation of Problem \eqref{problem4444} is locally Lipschitz function and is approximated by	the  function ${\bf \overline F^0}_{\widehat Q_n} \left( \widetilde U_{\rho_n}, \frac{\partial \widetilde U_{\rho_n}}{\partial x} \right)$ in the first equation of Problem \eqref{reguburger} where
	\begin{align} \label{definitionF}
		{\bf \overline F^0}_{\widehat Q_n}(v,\widehat v) :=
		\begin{cases}
			\widehat Q_n^2, & \max\{v, \widehat v \} \in (\widehat Q_n,+\infty),\\
			v \widehat v , & \max\{v, \widehat v \} \in [-\widehat Q_n, \widehat Q_n],\\
			\widehat Q_n^2, & \max\{v, \widehat v \}  \in (-\infty,-\widehat Q_n).
		\end{cases}
	\end{align}	
	Here	the function $\widehat Q_n$ is increasing  function and  $\lim_{n \rightarrow +\infty} \widehat Q_n=+\infty$.
	For a sufficiently large $n >0$ such that $$ \widehat Q_n \geq \max \Big( \|\textbf{u}\|_{L^\infty((0,T);L^2(\Omega))} , \|\textbf{u}_x\|_{L^\infty((0,T);L^2(\Omega))}  \Big). $$
	We show  that $	{\bf \overline F^0}_{\widehat Q_n}$ is a  globally Lipschitz function by the following Lemma
	\begin{lemma} \label{lemma5.2}
		For any $ (v, \widehat v) \in \mathbb{R}^2, (w, \widehat w)\in \mathbb{R}^2 $, we obtain
		\begin{align}
			\Big| {\bf \overline F^0}_{\widehat Q_n}(v,\widehat v)- {\bf \overline F^0}_{\widehat Q_n}(w,\widehat w)\Big| \le \widehat Q_n \Big( |v-\widehat v|+ | w-\widehat w|\Big). \label{F}
		\end{align}
	\end{lemma}

	\begin{proof}
		We divide the proof  into 5 cases:\\
		{\bf Case 1}.  If $\max\{v, \widehat v \} < -\widehat Q_n$ and  $\max\{w, \widehat w \} < -\widehat Q_n$ then it is easy to see that $	  {\bf \overline F^0}_{\widehat Q_n}(v,\widehat v)-  {\bf \overline F^0}_{\widehat Q_n}(w,\widehat w)=0$.\\
		{\bf Case 2}.  If $\max\{v, \widehat v \} < -\widehat Q_n \le   \max\{w, \widehat w \} \le  \widehat Q_n$ then using triangle inequality, we get
		\begin{align}
			\Big| 	 {\bf \overline F^0}_{\widehat Q_n}(v,\widehat v)-  {\bf \overline F^0}_{\widehat Q_n}(w,\widehat w)\Big| &= \Big| \widehat Q_n^2- w \widehat w  \Big|= \Big|  \widehat Q_n \left( \widehat Q_n+w \right)-w  \left( \widehat Q_n+\widehat w \right)\Big| \nn\\
			&\le \widehat Q_n  | w+\widehat  Q_n |+  |w|  | \widehat w+\widehat Q_n | \nn\\
			&\le \widehat Q_n  \Big(  | w+\widehat  Q_n |  +  | \widehat w+\widehat Q_n | \Big) \le \widehat Q_n   \Big(   |w-v |+ |\widehat w-\widehat v |\Big).\nn
		\end{align}
		{\bf Case 3}.   If $\max\{v, \widehat v \} < -\widehat Q_n <  \widehat Q_n \le \max\{w, \widehat w \} $ then
		\begin{align*}
			\Big| 	 {\bf \overline F^0}_{\widehat Q_n}(v,\widehat v)-  {\bf \overline F^0}_{\widehat Q_n}(w,\widehat w)\Big| &=  \Big| \widehat Q_n^2- \widehat Q_n^2     \Big|=0.
		\end{align*}
		{\bf Case 4}.  If $ -\widehat Q_n<\max\{v, \widehat v \},~  \max\{w, \widehat w \}  \le \widehat Q_n$ then
		\begin{align*}
			\Big| 	 {\bf \overline F^0}_{\widehat Q_n}(v,\widehat v)-  {\bf \overline F^0}_{\widehat Q_n}(w,\widehat w)\Big| &=  \Big| 	 v \widehat v-  w \widehat w\Big|= \Big| 	 (v-w) \widehat v +  w (\widehat v- \widehat w)\Big|\nn\\
			&\le |\widehat v| |v-w|+ |w| |\widehat v-\widehat w|
			\le \widehat Q_n   \Big(   |v- \widehat v |+ | w-\widehat w |\Big).
		\end{align*}
		{\bf Case 5}.  If $\max\{v, \widehat v \} > \widehat Q_n$ and  $\max\{w, \widehat w \} > \widehat Q_n$ then	\begin{align*}
			\Big| 	 {\bf \overline F^0}_{\widehat Q_n}(v,\widehat v)-  {\bf \overline F^0}_{\widehat Q_n}(w,\widehat w)\Big| &=  \Big| \widehat Q_n^2- \widehat Q_n^2     \Big|=0.
		\end{align*}
		By all  cases above, we complete the proof of Lemma \eqref{lemma5.2}.
	\end{proof}
	
	\section{ Regularized solutions for backward problem for  Burgers' equation}\label{section3-main}
	
	Our main result in this paper is stated as follows
	
	\begin{theorem} \label{burger}
		Let  the functions
		$H \in \mathcal{H}^{\mu_0}(\Omega) $ and $A, G \in L^\infty (0,T;\mathcal{H}^{\mu_0}(\Omega) )$, for $\mu_0 > \frac{1}{2}$.  Then problem \eqref{reguburger}  has unique solution ${\widetilde U_{\rho_n} } \in C([0,T];L^2(\Omega))$.
		Assume that Problem \eqref{problem4444} has unique solution ${\bf u} \in \kg $ . 	
		Let us choose $\widehat Q_{ n}$ such that
		\begin{equation}
			\lim_{n \to +\infty} \exp \Big(\frac{16 |\widehat Q_n|^2T}{A_1-A_0} \Big)    \max \Bigg( e^{2 \rho_n T}    \beta_{ n}^{1/2} n^{-4\mu },   e^{2 \rho_n T} \beta_{ n}^{-\mu_0}  , \rho_n^{-2\gamma}  \Bigg)  =0.
		\end{equation}
		Then for $n$ large enough,
		${\bf E}\|\widetilde U_{\rho_n,\beta_n}(x,t)-\textbf{u}(x,t)\|^2_{L^2(\Omega)}     $ is of order
		\begin{equation}
			\exp \Big(\frac{16 |\widehat Q_n|^2T}{A_1-A_0} \Big) e^{-2 \kappa_n t}  \max \Bigg( e^{2 \rho_n T}    \beta_{ n}^{1/2} n^{-4\mu },   e^{2 \rho_n T} \beta_{ n}^{-\mu_0}  , \rho_n^{-2\gamma}  \Bigg).
		\end{equation}
	\end{theorem}

	\begin{proof}
		Denote by
		\begin{align}
			B(x,t)=A_1 -A(x,t), \quad \overline B_{\beta_n}(x,t)= A_1-  \overline A_{\beta_n}(x,t).
		\end{align}
		The first equation of Problem \eqref{problem4444} can be  written as
		\begin{align}
			\frac{\partial {\bf u} }{\partial t} + \Big( \overline B_{\beta_n}(x,t) \frac{\partial {\bf u} }{\partial x}\Big)_x  &= {\bf u} {\bf u}_x+ \Big( \left(\overline B_{\beta_n}(x,t)- B(x,t)\right)\frac{\partial {\bf u} }{\partial x}\Big)_x \nn\\
			&+ A_1 \Delta {\bf u}+ G(x,t)
		\end{align}
		and 	the first equation of Problem \eqref{reguburger}  is rewritten as
		\begin{align}
			\frac{\partial \widetilde U_{\rho_n,\beta_n} }{\partial t} + \Big( \overline B_{\beta_n}(x,t) \frac{\partial \widetilde U_{\rho_n,\beta_n}  }{\partial x}\Big)_x & = 	{\bf \overline F^0}_{\widehat Q_n} \left( \widetilde U_{\rho_n,\beta_n}, \frac{\partial \widetilde U_{\rho_n,\beta_n}}{\partial x} \right)+ \Big( \left(\overline B_{\beta_n}(x,t)- B(x,t)\right)\frac{\partial \widetilde U_{\rho_n,\beta_n} }{\partial x}\Big)_x \nn\\
			&+	\mathbb{\bf P}_{n} \widetilde U_{\rho_n,\beta_n} +\overline G_{\rho_n} (x,t).
		\end{align}
		For $\kappa_n>0$, we put $$ \mathbf{Y}_{\rho_n,\beta_n}(x,t)=e^{\kappa_n(t-T)}\Big[ \widetilde U_{\rho_n,\beta_n}(x,t)-\textbf{u}(x,t)\Big].$$
 Then the last two equations,  and  a simple computation gives
		\begin{eqnarray}
			\begin{aligned}
				\frac{\partial \mathbf{Y}_{\rho_n,\beta_n}}{\partial t} & + \Big(\overline B_{\rho_n} \frac{\partial \mathbf{Y}_{\rho_n,\beta_n} }{\partial x}\Big)_x-\kappa_n \mathbf{Y}_{\rho_n,\beta_n} \nn\\
				&=
				\mathbb{\bf P}_{n} \mathbf{Y}_{\rho_n,\beta_n}
				- e^{\kappa_n(t-T)} \left( {\bf P}_{\rho_n}-{\bf P} \right)  \textbf{u} - e^{\kappa_n(t-T)}\Big( \left(\overline B_{\beta_n}(x,t)- B(x,t)\right)\frac{\partial {\bf u} }{\partial x}\Big)_x \nn\\
				&\quad + e^{\kappa_n (t-T)}\lf[	{\bf \overline F^0}_{\widehat Q_n} \left( \widetilde U_{\rho_n,\beta_n}, \frac{\partial \widetilde U_{\rho_n,\beta_n}}{\partial x} \right)- {\bf u} {\bf u}_x\rt]  + e^{\kappa_n (t-T)}\lf[\overline G_{\beta_n}(x,t))- G(x,t)\rt]
			\end{aligned}
		\end{eqnarray}
		and $\mathbf{Y}_{\rho_n,\beta_n}|_{\partial \Omega}=0, ~ \mathbf{Y}_{\rho_n,\beta_n}(x,T)=\overline H_{\beta_n}(x)- H(x) $.\\
		By taking the inner product of  the  two sides of the last equality with $\mathbf{Y}_{\rho_n,\beta_n}$ and noting the  equality $$\int_\Omega  \Big(\overline B_{\beta_n} \frac{\partial \mathbf{Y}_{\rho_n,\beta_n} }{\partial x}\Big)_x \mathbf{Y}_{\rho_n,\beta_n} dx=-\int_\Omega \overline B_{\beta_n}(x,t) \Big|\frac{\partial \mathbf{Y}_{\rho_n,\beta_n} }{\partial x}\Big|^2 dx,$$ one deduces that
		\begin{eqnarray}
			\begin{aligned} \label{3J}
				\frac{1}{2} \frac{d}{dt}  \|\mathbf{Y}_{\rho_n,\beta_n}(\cdot, t)\|^2_{L^2(\Omega)} &- \int_\Omega \overline B_{\beta_n}(x,t) \Big|\frac{\partial \mathbf{Y}_{\rho_n,\beta_n} }{\partial x}\Big|^2 dx-\kappa_n \|\mathbf{Y}_{\rho_n,\beta_n}(\cdot, t)\|^2_{L^2(\Omega)}\nn\\
				& = \underbrace{ \Big<	\mathbb{\bf P}_{n} \mathbf{Y}_{\rho_n,\beta_n}, \mathbf{Y}_{\rho_n,\beta_n} \big>_{L^2(\Omega)}}_{=:\widetilde{\mathcal{J}}_{12,n}} + \underbrace{ \Big< e^{\kappa_n(t-T)} \left( {\bf P}_{\rho_n}-{\bf P} \right)  \textbf{u}, \mathbf{Y}_{\rho_n,\beta_n} \Big>_{L^2(\Omega)}}_{=:\widetilde{\mathcal{J}}_{13,n}} \nn\\
				&+ \underbrace{\Big< - e^{\kappa_n(t-T)}\Big( \left(\overline B_{\beta_n}(x,t)- B(x,t)\right)\frac{\partial {\bf u} }{\partial x}\Big)_x, \mathbf{Y}_{\rho_n,\beta_n} \Big>_{L^2(\Omega)}}_{=:\widetilde{\mathcal{J}}_{14,n}} \nn\\
				&+\underbrace{\Big< e^{\kappa_n (t-T)}\lf[	{\bf \overline F^0}_{\widehat Q_n} \left( \widetilde U_{\rho_n,\beta_n}, \frac{\partial \widetilde U_{\rho_n}}{\partial x} \right)- {\bf u} {\bf u}_x\rt] , \mathbf{Y}_{\rho_n,\beta_n} \Big>_{L^2(\Omega)}}_{=:\widetilde{\mathcal{J}}_{15,n}}\nn\\
				&+\underbrace{\Big<e^{\kappa_n (t-T)}\lf[\overline G_{\beta_n}(x,t))- G(x,t)\rt], \mathbf{Y}_{\rho_n,\beta_n} \Big>_{L^2(\Omega)}}_{=:\widetilde{\mathcal{J}}_{16,n}}	.
			\end{aligned}
		\end{eqnarray}
		For   $\widetilde{\mathcal{J}}_{12,n}$, we have the following
		\begin{align} \label{J12}
			\big|\widetilde{\mathcal{J}}_{12,n} \big| &\leq \| \mathbb{\bf P}_{n} \mathbf{Y}_{\rho_n}\|_{L^2(\Omega)}  \| \mathbf{Y}_{\rho_n} (\cdot,t) \|_{L^2(\Omega)}\le  \rho_n \| \mathbf{Y}_{\rho_n} (\cdot,t) \|_{L^2(\Omega)}^2,
		\end{align}
		where we used inequality \eqref{estimateP1}.
		And for $\widetilde{\mathcal{J}}_{13,n}$, using Cauchy-Schwartz and \eqref{estimateP2}, we have the following upper bound
		\begin{align} \label{J13}
			\big|\widetilde{\mathcal{J}}_{13,n}\big|
			& \leq \frac{1}{2} e^{2\kappa_n(t-T)}  A_1^2  \rho_n^{-2\gamma}e^{-2T\rho_n } \|\textbf{u}\|_{\kg}^2  +\frac{1}{2} \|\mathbf{Y}_{\rho_n,\beta_n}(\cdot,t)\|^2_{L^2(\Omega)}.
		\end{align}
		The  Cauchy-Schwartz inequality leads to the following estimation
		\begin{eqnarray}
			\begin{aligned} \label{J3}
				\big|\widetilde{\mathcal{J}}_{14,n}\big| &=\lf|\Big< - e^{\kappa_n(t-T)}\Big( \left(\overline B_{\beta_n}(x,t)- B(x,t)\right)\frac{\partial {\bf u} }{\partial x}\Big)_x, \mathbf{Y}_{\rho_n,\beta_n} \Big>_{L^2(\Omega)}\rt| \nn\\
				&=\lf|\Big<- e^{\kappa_n(t-T)}\Big( \left(\overline B_{\beta_n}(x,t)- B(x,t)\right)\frac{\partial {\bf u} }{\partial x}\Big), \frac{\partial \mathbf{Y}_{\rho_n,\beta_n} }{\partial x} \Big>_{L^2(\Omega)}\rt| \nn\\
				&\leq \frac{ e^{2\kappa_n(t-T)}}{2(A_1-A_0)}\|\overline B_{\rho_n}(.,t)- B(.,t)\|_{L^2(\Omega)}^2\lf\| \frac{\partial {\bf u} }{\partial x} (\cdot,t) \rt\|_{L_2(\Omega)}^2 + \frac{A_1-A_0}{2}  \int_{\Omega}  \lf|\frac{\partial \mathbf{Y}_{\rho_n,\beta_n} }{\partial x}\rt|^2 dx\nn\\
				&\le \frac{ \|\overline B_{\beta_n}(.,t)- B(.,t)\|_{L^2(\Omega)}^2}{2(A_1-A_0)}\lf\| \textbf{u} (\cdot,t) \rt\|_{H_0^1(\Omega)}^2 + \frac{ A_1-A_0}{2}  \left\| \frac{\partial \mathbf{Y}_{\rho_n,\beta_n} }{\partial x} \right\|_{L^2(\Omega)}^2.
			\end{aligned}
		\end{eqnarray}
		For  $\widetilde{\mathcal{J}}_{15,n} $ , we note that  ${\bf \overline F^0}_{\widehat Q_n}({\bf u}, {\bf u}_x)= {\bf u}{\bf u}_x $ and
		thanks to \eqref{F}, we obtain
		\begin{eqnarray}
			\begin{aligned}
				\Big\|  	{\bf \overline F^0}_{\widehat Q_n} \left( \widetilde U_{\rho_n,\beta_n}, \frac{\partial \widetilde U_{\rho_n,\beta_n}}{\partial x} \right)-   {\bf u}{\bf u}_x \Big\|_{L^2(\Omega)}
				&= \Big\|  	{\bf \overline F^0}_{\widehat Q_n} \left( \widetilde U_{\rho_n,\beta_n}, \frac{\partial \widetilde U_{\rho_n,\beta_n}}{\partial x} \right)-  {\bf \overline F^0}_{\widehat Q_n}({\bf u}, {\bf u}_x) \Big\|_{L^2(\Omega)}\nn\\
				&\le  \widehat Q_n \Big(  \left\| \widetilde U_{\rho_n,\beta_n}-{\bf u}  \right\|_{L^2(\Omega)} +\left\| \frac{\partial \widetilde U_{\rho_n,\beta_n}}{\partial x}-{\bf u}_x  \right\|_{L^2(\Omega)} \Big)\nn\\
				&= e^{\kappa_n(T-t)} \widehat Q_n  \Big( \left\| \mathbf{Y}_{\rho_n,\beta_n} \right\|_{L^2(\Omega)} +\left\| \frac{\partial \mathbf{Y}_{\rho_n,\beta_n}}{\partial x} \right\|_{L^2(\Omega)} \Big) \nn\\
				&\le 2e^{\kappa_n(T-t)} \widehat Q_n \left\| \frac{\partial \mathbf{Y}_{\rho_n,\beta_n}}{\partial x} \right\|_{L^2(\Omega)},
			\end{aligned}
		\end{eqnarray}
		where we note that $\left\| \mathbf{Y}_{\rho_n,\beta_n} \right\|_{L^2(\Omega)}  \le \left\| \frac{\partial \mathbf{Y}_{\rho_n,\beta_n}}{\partial x} \right\|_{L^2(\Omega)} $.
		This implies that
		\begin{eqnarray}
			\begin{aligned} \label{J4}
				\big|\widetilde{\mathcal{J}}_{15,n}\big|&= \Big| \Big< e^{\kappa_n (t-T)}\lf[	{\bf \overline F^0}_{\widehat Q_n} \left( \widetilde U_{\rho_n,\beta_n}, \frac{\partial \widetilde U_{\rho_n,\beta_n}}{\partial x} \right)- {\bf u} {\bf u}_x\rt] , \mathbf{Y}_{\rho_n,\beta_n} \Big>_{L^2(\Omega)}\Big|\nn\\
				&\le e^{2\kappa_n (t-T)} \frac{A_1-A_0}{8|\widehat Q_n|^2} \Big\|  	{\bf \overline F^0}_{\widehat Q_n} \left( \widetilde U_{\rho_n,\beta_n}, \frac{\partial \widetilde U_{\rho_n,\beta_n}}{\partial x} \right)-   {\bf u}{\bf u}_x \Big\|_{L^2(\Omega)}^2+\frac{8 |\widehat Q_n|^2}{A_1-A_0}\|\mathbf{Y}_{\rho_n,\beta_n}(\cdot,t)\|^2_{L^2(\Omega)}\nn\\
				&\le \frac{ A_1-A_0}{2}  \left\| \frac{\partial \mathbf{Y}_{\rho_n,\beta_n}}{\partial x} \right\|_{L^2(\Omega)}^2+\frac{8 |\widehat Q_n|^2}{A_1-A_0}\|\mathbf{Y}_{\rho_n,\beta_n}(\cdot,t)\|^2_{L^2(\Omega)}.
			\end{aligned}
		\end{eqnarray}
		The term $\big|\widetilde{\mathcal{J}}_{16,n}\big|$ can be bounded by
		\begin{eqnarray}
			\begin{aligned} \label{J4}
				\big|\widetilde{\mathcal{J}}_{16,n}\big|&= \Big| \Big<e^{\kappa_n (t-T)}\lf[\overline G_{\beta_n}(.,t))- G(.,t)\rt], \mathbf{Y}_{\rho_n,\beta_n} \Big>_{L^2(\Omega)}\Big|\nn\\
				& \leq \frac{1}{2}e^{2\kappa_n (t-T)} \lf\|\overline G_{\beta_n}(.,t))- G(.,t)\rt\|^2_{L^2(\Omega)}+ \frac{1}{2}\|\mathbf{Y}_{\rho_n,\beta_n}\|_{L^2(\Omega)}.
			\end{aligned}
		\end{eqnarray}
		Combining all the previous estimates, we get
		\begin{eqnarray}
			\begin{aligned}
				&\frac{d}{dt}  \|\mathbf{Y}_{\rho_n,\beta_n}(\cdot, t)\|^2_{L^2(\Omega)} - 2\int_\Omega \overline B_{\beta_n}(x,t) \Big|\frac{\partial \mathbf{Y}_{\rho_n,\beta_n} }{\partial x}\Big|^2 dx-2\kappa_n \|\mathbf{Y}_{\rho_n,\beta_n}(\cdot, t)\|^2_{L^2(\Omega)}\nn\\
				&\ge -2 \rho_n \| \mathbf{Y}_{\rho_n,\beta_n} (\cdot,t) \|_{L^2(\Omega)}^2- e^{2\kappa_n(t-T)}  A_1^2  \rho_n^{-2\gamma}e^{-2T\rho_n } \|\textbf{u}\|_{\kg}^2 \nn\\
				&  - \|\mathbf{Y}_{\rho_n,\beta_n}(\cdot,t)\|^2_{L^2(\Omega)}
				-\frac{ \|\overline B_{\beta_n}(.,t)- B(.,t)\|_{L^2(\Omega)}^2}{(A_1-A_0)}\lf\| \textbf{u} (\cdot,t) \rt\|_{H_0^1(\Omega)}^2 \nn\\
				&- 2( A_1-A_0)  \left\| \frac{\partial \mathbf{Y}_{\rho_n,\beta_n}}{\partial x} \right\|_{L^2(\Omega)}^2-\frac{16 |\widehat Q_n|^2}{A_1-A_0}\|\mathbf{Y}_{\rho_n,\beta_n}(\cdot,t)\|^2_{L^2(\Omega)}\nn\\
				&-e^{2\kappa_n (t-T)} \lf\|\overline G_{\beta_n}(.,t))- G(.,t)\rt\|^2_{L^2(\Omega)}- \|\mathbf{Y}_{\rho_n,\beta_n}\|^2_{L^2(\Omega)}.
			\end{aligned}
		\end{eqnarray}
		By taking the integral from $t$ to $T$ and by a simple calculation yields
		\begin{eqnarray}
			\begin{aligned}
				&\|\mathbf{Y}_{\rho_n,\beta_n}(\cdot, T)\|^2_{L^2(\Omega)}  -\|\mathbf{Y}_{\rho_n,\beta_n}(\cdot, t)\|^2_{L^2(\Omega)} \nn\\
				&+ \int_t^T \lf( A_1^2  \rho_n^{-2\gamma}e^{-2T\rho_n } \|\textbf{u}\|_{\kg}^2 +  \frac{ \|\overline B_{\beta_n}(x,t)- B(x,t)\|_{L^2(\Omega)}^2}{(A_1-A_0)} \lf\| \textbf{u}  \rt\|_{L^\infty((0,T);H_0^1(\Omega))}^2 \rt) ds \nn\\
				&\geq 2 \int_t^T \int_\Omega \Big(\overline B_{\rho_n}(x,s)-(A_1-A_0)\Big) \left\| \frac{\partial \mathbf{Y}_{\rho_n,\beta_n}}{\partial x} \right\|_{L^2(\Omega)}^2 dx  ds \nn\\
				&\quad  +\int_t^T \lf(2\kappa_n- 2\rho_n- \frac{16 |\widehat Q_n|^2}{A_1-A_0} -2 \rt)\|\mathbf{Y}_{\rho_n,\beta_n}(\cdot, s) \|^2_{L^2(\Omega)} ds\nn\\
				&\quad \quad \quad \quad \quad-Te^{2\kappa_n (t-T)} \lf\|\overline G_{\beta_n}(.,t))- G(.,t)\rt\|^2_{L^\infty(0,T;L^2(\Omega))}\nn\\
				&\geq \int_t^T \lf(2\kappa_n-2 \rho_n- \frac{16 |\widehat Q_n|^2}{A_1-A_0}-2 \rt)\|\mathbf{Y}_{\rho_n,\beta_n}(\cdot, s) \|^2_{L^2(\Omega)} ds\nn\\
				&\quad -Te^{2\kappa_n (t-T)} \lf\|\overline G_{\beta_n}(.,t))- G(.,t)\rt\|^2_{L^\infty(0,T;L^2(\Omega))}.
			\end{aligned}
		\end{eqnarray}
		where we used the fact that  $$\overline B_{\beta_n}(x,s) = A_1- \overline A_{\rho_n}(x,s) \ge A_1-A_0.$$
		Let us choose $\kappa_n=\rho_n$ then we
		obtain
		\begin{eqnarray}
			\begin{aligned}
				&e^{2 \kappa_n (t-T)}	{\bf E}\|\widetilde U_{\rho_n,\beta_n}(.,t)-\textbf{u}(.,t)\|^2_{L^2(\Omega)}   \nn\\
				&\quad \quad \quad \le 	{\bf E} \| \overline H_{\beta_n}- H \|^2_{L^2(\Omega)}+ T A_1^2  \rho_n^{-2\gamma}e^{-2T\rho_n } \|\textbf{u}\|_{\kg}^2\nn\\
				&\quad \quad \quad+T{\bf E} \lf\|\overline G_{\beta_n}(.,t))- G(.,t)\rt\|^2_{L^\infty(0,T;L^2(\Omega))}  \nn\\
				&\quad \quad \quad+  \frac{ {\bf E} \|\overline B_{\beta_n}(.,t)- B(.,t)\|_{L^\infty(0,T;L^2(\Omega))}^2}{(A_1-A_0)} \lf\| \textbf{u}  \rt\|_{L^\infty((0,T);H_0^1(\Omega))}^2  \nn\\
				&\quad \quad \quad + \Big(\frac{16 |\widehat Q_n|^2}{A_1-A_0} +2\Big) \int_t^T e^{2 \kappa_n (s-T)}	{\bf E}\|\widetilde U_{\rho_n,\beta_n}(.,s)-\textbf{u}(.,s)\|^2_{L^2(\Omega)}ds .
			\end{aligned}
		\end{eqnarray}
		Multiplying both sides of the last inequality by $e^{2 \kappa_n T}$, we obtain
		\begin{eqnarray}
			\begin{aligned}
				&e^{2 \kappa_n t}	{\bf E}\|\widetilde U_{\rho_n,\beta_n}(.,t)-\textbf{u}(.,t)\|^2_{L^2(\Omega)}   \nn\\
				& \le 	e^{2 \kappa_n T}{\bf E} \| \overline H_{\beta_n}- H \|^2_{L^2(\Omega)}+ T A_1^2  \rho_n^{-2\gamma}  \|\textbf{u}\|_{\kg}^2\nn\\
				&+T e^{2 \kappa_n T} {\bf E} \lf\|\overline G_{\beta_n}- G\rt\|^2_{L^\infty(0,T;L^2(\Omega))}  \nn\\
				&+ e^{2 \kappa_n T} \frac{ {\bf E} \|\overline B_{\beta_n}- B\|_{L^\infty(0,T;L^2(\Omega))}^2}{(A_1-A_0)} \lf\| \textbf{u}  \rt\|_{L^\infty((0,T);H_0^1(\Omega))}^2  \nn\\
				& + \Big(\frac{16 |\widehat Q_n|^2}{A_1-A_0} +2\Big) \int_t^T e^{2s \kappa_n }	{\bf E}\|\widetilde U_{\rho_n,\beta_n}(.,s)-\textbf{u}(.,s)\|^2_{L^2(\Omega)}ds .
			\end{aligned}
		\end{eqnarray}
		Applying Gronwall's inequality, we deduce that
		\begin{eqnarray}
			\begin{aligned}
				&{\bf E}\|\widetilde U_{\rho_n,\beta_n}(x,t)-\textbf{u}(x,t)\|^2_{L^2(\Omega)}  \le \exp \Big(\frac{16 |\widehat Q_n|^2(T-t)}{A_1-A_0} +2(T-t)\Big) e^{-2 \kappa_n t} B'  \label{B}
			\end{aligned}
		\end{eqnarray}
		where
		\begin{eqnarray}
			\begin{aligned}
				B'&= e^{2 \kappa_n T}{\bf E} \| \overline H_{\beta_n}- H \|^2_{L^2(\Omega)}+ T A_1^2  \rho_n^{-2\gamma}  \|\textbf{u}\|_{\kg}^2 \nn\\
				&+T e^{2 \kappa_n T} {\bf E} \lf\|\overline G_{\beta_n}- G\rt\|^2_{L^\infty(0,T;L^2(\Omega))}+e^{2 \kappa_n T} \frac{ {\bf E} \|\overline B_{\beta_n}- B\|_{L^\infty(0,T;L^2(\Omega))}^2}{(A_1-A_0)} \lf\| \textbf{u}  \rt\|_{L^\infty((0,T);H_0^1(\Omega))}^2.
			\end{aligned}
		\end{eqnarray}
		Thanks to  Theorem \ref{theorem2.1}, we have that
		$$
		{\bf E}	\Big\| \widehat H_{\beta_{\bf n}}-H \Big\|_{L^2(\Omega)}^2 +T {\bf E}  \Big\| \widehat G_{\beta_{\bf n}}-G \Big\|_{L^\infty(0,T;L^2(\Omega))}^2+ {\bf E} \|\overline B_{\beta_n}- B\|_{L^\infty(0,T;L^2(\Omega))}^2
		$$
		is of order
		$ \max \Big( \beta_{ n}^{1/2} n^{-4\mu },  \beta_{ n}^{-\mu_0}   \Big) $ for any $\mu > \frac{1}{2}$. This together with \eqref{B} implies that ${\bf E}\|\widetilde U_{\rho_n,\beta_n}(.,t)-\textbf{u}(.,t)\|^2_{L^2(\Omega)}     $ is of order
		\begin{equation}
			\exp \Big(\frac{16 |\widehat Q_n|^2T}{A_1-A_0} \Big) e^{-2 \kappa_n t}   \max \Bigg( e^{2 \rho_n T}    \beta_{ n}^{1/2} n^{-4\mu },   e^{2 \rho_n T} \beta_{ n}^{-\mu_0}  , \rho_n^{-2\gamma}  \Bigg).
		\end{equation}
		
	\end{proof}

\end{document}